\documentclass[a4paper,11pt]{amsart}

\usepackage{latexsym}
\usepackage{amssymb,amsmath, amsthm}
\usepackage{graphics}
\usepackage{url}
\usepackage{booktabs}  

\newtheorem{theorem}{Theorem}[section]
\newtheorem{corollary}[theorem]{Corollary}
\newtheorem{proposition}[theorem]{Proposition}
\newtheorem{conjecture}[theorem]{Conjecture}
\newtheorem{problem}[theorem]{Problem}

\newtheorem{lemma}[theorem]{Lemma}

\theoremstyle{definition}

\newtheorem{example}[theorem]{Example}

\newtheorem*{strategy}{Strategy}
\newtheorem{thmmain}{Theorem}

\input{xy}
\xyoption{all}

\newcommand{\N}{\mathbf{N}}

\renewcommand{\epsilon}{\varepsilon}

\newcounter{thmlistcnt}
	{\setcounter{thmlistcnt}{0}%
	\begin{list}{\emph{(\roman{thmlistcnt})}}{%
		\usecounter{thmlistcnt}%
		\setlength{\topsep}{0pt}%
		\setlength{\leftmargin}{36pt}%
		\setlength{\labelwidth}{17pt}%
		\setlength{\itemsep}{0pt}%
		\setlength{\itemindent}{0pt}}%
	}%
	{\end{list}}%

\renewcommand{\c}[1]{#1^{\raisebox{1pt}{$\hskip0.5pt\scriptstyle\star$}}}
\newcommand{\n}[1]{#1}

\usepackage{caption}
\begin{document}
\title[Searching for knights and spies]{Searching for 
knights and spies: a
majority/minority game}
\date{\today}
\author{Mark Wildon}
\maketitle
\thispagestyle{empty}

\begin{abstract}
There are $n$ people, each of whom
is either a knight or a spy. 
It is known that at least $k$ knights are present, where $n/2 < k < n$.
Knights always tell the truth.
We consider both spies who always lie and spies who answer as they see fit. 
This paper determines the number of questions required to find a spy 
or prove that everyone in the room is a knight.
We also determine
the minimum number of questions needed to find at least one person's identity,
or a nominated person's identity, 
or to find a spy (under the assumption that a spy is present).
For spies who always lie, we prove that
these searching problems, and the problem of finding a knight, can be solved by a simultaneous optimal strategy. We also
give some computational results on the problem of finding all identities
when spies always lie, and end by stating some open problems.
\end{abstract}

\section{Introduction}

In a room there are $n$ people, numbered from $1$ up to $n$.
Each person is either a \emph{knight} or a \emph{spy}, and will answer any question of the form
\begin{center}`Person $x$, is person $y$ a spy?'
\end{center}
Knights always answer truthfully. We shall consider both 
spies who always lie, and unconstrained spies 
who  lie or tell the truth as they see fit. We work in
the adaptive model in which future questions may be chosen
in the light of the answers to earlier questions. 
We always assume that knights are in a strict majority, since otherwise, even if every
permitted question is asked, it may be impossible to be certain of anyone's identity.

In this paper we determine the number of questions that are necessary and sufficient to find a spy,
or to find at least one person's identity, or to find an identity of a specific person,
nominated in advance. We also survey the existing work on
the problems of finding a knight or finding everyone's identity, and prove two theorems
showing the extent to which the problems considered in this paper admit a common solution. 
In the final section we state some open problems suggested by the five main theorems
and present some computational results on the problem of finding all identities when spies always lie.
A recurring theme is that early accusations are  very helpful when finding
spies, since at least one of the  people involved must be a spy.

We work in the general setting, also considered in \cite{Aigner}, 
where it is known that at least $k$ of the $n$ people are knights, where $n/2 < k < n$.
Throughout~$n$ and $k$ have these meanings.
For 
spies who always lie, let
\begin{itemize}
\item[$\bullet$] $\n{T}_L(n,k)$ be the number of questions that
are necessary and sufficient either to
identify a spy, or to make a correct claim that everyone in the
room is a knight;
\item[$\bullet$] $\c{T}_L(n,k)$ be the number of questions that are necessary
and sufficient
to identify a spy, if it is known
that at least one spy is present. 
\end{itemize}
Our first main result is proved in \S\ref{sec:liars}.

\begin{thmmain}\label{thm:liars}
Let $n = q(n-k+1) + r$ where $0 \le r \le n-k$. Then
\begin{align*}
\n{T}_L(n,k) &= \begin{cases} n-q  + 1 & \text{if $r=0$} \\ 
n-q & \text{if $r = 1$} \\
n-q & \text{if $r \ge 2$}  \end{cases}          
\intertext{and}\\[-30pt]
\c{T}_L(n,k) &= \begin{cases} n-q   & \text{if $r=0$} \\
                                    n-q & \text{if $r=1$} \\    
                            n-q-1  & \text{if $r \ge 2$.}
              \end{cases} 
\end{align*}
with the single exception that $\c{T}_L(5,3) = 4$.
\end{thmmain}

In particular, we have $\n{T}_L(n,k) = \c{T}_L(n,k) + 1$ except when
$(n,k) = (5,3)$ or $n = q(n-k+1) + 1$ for some $q \in \N$; in these 
cases equality holds.

Let $\c{T}_S(n,k)$ and $\n{T}_S(n,k)$ be the analogously defined numbers if spies are unconstrained.
In this setting we prove the following result in \S\ref{sec:spies}.

\begin{thmmain}\label{thm:spies}
We have $\c{T}_S(n,k) = n-1$ and $\n{T}_S(n,k) = n$.
\end{thmmain}

Note that, in contrast to $\c{T}_L(n,k)$ and $\n{T}_L(n,k)$, the numbers
$\c{T}_S(n,k)$ and $\n{T}_S(n,k)$ are independent of $k$.
Theorems~\ref{thm:liars} and~\ref{thm:spies} are proved in \S\ref{sec:liars} and~\S\ref{sec:spies} below.
The proof of the lower bound needed for Theorem~\ref{thm:liars} has some features
in common with Theorem~4 in \cite{Aigner}: we connect these results in \S\ref{subsec:all}.

To state the third main theorem we must introduce eight further numbers.
For 
spies who always lie, let 
\begin{itemize}
\item $K_L(n,k)$ be the number of questions that are necessary and sufficient
to find a knight;
\item $E_L(n,k)$ be the number of questions that are necessary and sufficient
to find at least one person's identity; 
\item $N_L(n,k)$ be the 
number of questions that are necessary and sufficient to identify Person $1$.
\end{itemize}
Let $K_S(n,k)$, $E_S(n,k)$ and $N_S(n,k)$ 
be the analogously defined numbers when spies are unconstrained.
Let 
$\c{N}_L(n,k)$
and $\c{N}_S(n,k)$ be the numbers corresponding to $N_L(n,k)$ and $N_S(n,k)$
defined on the assumption that a spy is present.
Let $B(s)$ be the number of $1$s in the binary expansion of $s \in \N$.
In~\S\ref{sec:knights} we prove the following theorem.

\begin{thmmain}\label{thm:knights} We have
\[ K_S(n,k) = K_L(n,k) = E_S(n,k) = E_L(n,k) = 2(n-k) - B(n-k) \]
and the same holds for the corresponding numbers defined on the assumption that a spy
is present. Moreover
\[  N_S(n,k) = N_L(n,k) = \c{N}_S(n,k) = \c{N}_L(n,k) = 2(n-k) - B(n-k) + 1. \]
with the exception that $\c{N}_L(n,k) = 2(n-k) - B(n-k)$ when $n = 2^{e+1}+1$
and $k = 2^{e}+1$ for some $e \in \N$.
\end{thmmain}

The numbers $K_S(n,k)$ and $K_L(n,k)$ have already been studied.
If spies always lie,
then one person supports
another if and only if they are of the same type and accuses if and only
if they are of different types. 
Therefore finding a knight when spies always lie 
is equivalent to the \emph{majority game} of
identifying a ball of a majority colour
in a collection of $n$ balls coloured with two colours, 
using only binary comparisons between pairs of balls that result in
the information `same colour' or
`different colours'.
For an odd number of balls, the relevant part of Theorem~\ref{thm:knights} is 
that $K_L(2k-1,k) = 2(k-1) - B(k-1)$.
This result was first proved by 
Saks and Werman in~\cite{SaksWerman}. A particularly elegant
proof was later given by Alonso, Reingold and Schott
in~\cite{AlonsoEtAl}.
In Theorem 6 of \cite{Aigner}, Aigner
adapts the questioning strategy introduced in
\cite{SaksWerman} to show that $K_S(n,k) \le 2(n-k) - B(n-k)$.
We recall Aigner's questioning strategy and the proof of this result in \S 2 below.
Aigner also claims a proof, based on Lemma~5.1 in \cite{Wiener}, that~$K_L(n,k) \ge 2(n-k) - B(n-k)$.
A flaw in these proofs
was pointed out in \cite{BritnellWildonMajority}, and a correct proof was given.
It is obvious that $K_S(n,k) \ge K_L(n,k)$, so it follows that
$K_S(n,k) = K_L(n,k) = 2(n-k) - B(n-k)$,
giving part of Theorem~\ref{thm:knights}. 

It is 
natural to ask whether when there are questioning strategies
that solve the searching problems considered in Theorems~\ref{thm:liars},~\ref{thm:spies}
and~\ref{thm:knights} simultaneously. In \S\ref{sec:liarscomb} we prove
that, perhaps surprisingly, there is such a strategy when spies always lie.
Define $K(n,k) = 2(n-k) - B(n,k)$.

\begin{thmmain}\label{thm:liarscomb}
Suppose that  spies always lie. 
There is a questioning strategy that will find
a knight by question $K(n,k)$, find Person $1$'s identity 
by question $K(n,k) + 1$ and by question $\n{T}_L(n,k)$
either find a spy or prove that everyone in the room is a knight. 
Moreover if a spy is known to be present then a spy will be
found by question $\c{T}_L(n,k)$.
\end{thmmain}

We also show in \S\ref{sec:liarscomb}
that by asking further questions it is possible to determine
all identities by question $n-1$.
In the important special case where $n=2k-1$, so all
that is known is that knights are in a strict majority, there are $2^{n-1}$ possible
sets of spies, and so $n-1$ questions are obviously necessary to determine all identities.
Thus in this case, 
all five problems admit a simultaneous optimal solution.
We make some further remarks on finding all identities when spies always 
lie, and ask a natural question suggested by Theorem~\ref{thm:liarscomb},
 in the final section of this paper. 

For unconstrained spies it is impossible in general to solve the four problems
by a single strategy. The following theorem, proved in \S\ref{sec:spiescomb}, shows one obstruction.

\begin{thmmain}\label{thm:spiescomb}
Suppose that 
spies are unconstrained.
There is a questioning strategy that will find a knight by
question $K(n,k) + 1$, find Person $1$'s identity by question $K(n,k) + 2$,
and by question $\n{T}_S(n,k) = n$ either find a spy, or prove
that everyone in the room is a knight. 
Moreover, if a spy is known to be present,
then a spy will be found by question $\c{T}_S(n,k) = n-1$.
When $n=7$ and $k=4$ there is no questioning strategy that will both find a knight by
question $K_S(7,4) = 4$ and find a spy by question $\c{T}_S(n,k) = 6$.
\end{thmmain}

There is an adversarial game associated to each of our theorems,
in which questions are put by an \emph{Interrogator} and answers are
decided by a \emph{Spy Master}, whose task is to ensure that the Interrogator
asks at least the number of questions claimed to be necessary. 
We shall use
this game-playing setup without further comment. 
We represent
 positions part-way through a game by a 
 \emph{question graph},
with vertex set $\{1,2,\ldots, n\}$, in which 
there is a directed edge from $x$ to $y$
if Person $x$ has been asked about Person~$y$, labelled by Person $x$'s reply. 
We rule out loops by making the simplifying assumption that no-one is ever asked
for his own identity: such questions are clearly pointless.
In figures, accusations are shown by
dashed arrows and supportive statements by solid arrows.
Since each question reduces the number of components in the question
graph by at most one, it takes $n-c$ questions to form a question graph with
 $c$ or fewer components. We shall use this observation many times below.
 
 \subsection*{Outline}
We remind the reader of the structure of the paper: 
\S\ref{sec:BKH} gives a
basic strategy for finding a knight. In \S\ref{sec:liars}, \S\ref{sec:knights} and~\S\ref{sec:liarscomb} 
we prove Theorems~\ref{thm:liars},~\ref{thm:knights} and~\ref{thm:liarscomb}. 
In Theorems~\ref{thm:liars} and~\ref{thm:liarscomb} spies always lie, and this
is also the most important case for Theorem~\ref{thm:knights}. 
In \S\ref{sec:spies} and~\S\ref{sec:spiescomb}
we prove Theorems~\ref{thm:spies} and~\ref{thm:spiescomb} on unconstrained spies.
In \S 8 we give some  computational results and state some open problems.

\section{Binary Knight Hunt}\label{sec:BKH}

This questioning strategy was introduced in Theorem~6 of \cite{Aigner}. (The 
present name is the author's invention.)
We shall use variants of it 
in the proofs of Theorems~\ref{thm:liars},~\ref{thm:knights},
~\ref{thm:liarscomb} and~\ref{thm:spiescomb}. See Example~\ref{ex:ex}
for its use in the strategy used to prove Theorem~\ref{thm:liarscomb}.

\begin{strategy}[\bf Binary Knight Hunt]
The starting position is a set $P$ of people, none of whom has
been asked a question or asked about. After each question,
every component $C$ of the question graph that is contained in $P$ has 
a unique \emph{sink vertex} which
can be reached by a directed path from any other vertex in $C$.
To decide on a question:

\begin{itemize}
\item If the components in $P$
in which no accusation has been made all have different sizes, the
strategy terminates. 

\item Otherwise, the Interrogator 
chooses two components $C$ and $C'$ in $P$ of equal size in which no accusation has
been made. If $C$ has sink vertex $x$ and $C'$ has sink vertex $x'$, then
he asks Person $x$ about Person $x'$, forming a new component with sink vertex $x'$.
\end{itemize}
\end{strategy}

We call components in which an accusation has been made \emph{accusatory}. Since
anyone who supports a spy (either directly, or via a directed path of supportive
edges) is a spy, and each
accusatory component is formed by connecting two sink vertices in components of equal
size with no accusations, each accusatory
component contains at least as many spies as knights.

The Binary Knight Hunt is immediately effective when knights are in a strict majority in $P$.
Note that after each
question, each component in~$P$ has size a power of two.
Suppose that
when the strategy terminates, there are non-accusatory components of
distinct sizes $2^{b_1}, 2^{b_2}, \ldots, 2^{b_u}$ where $b_1 < \ldots < b_u$.
Since there are at least as many spies as knights in each accusatory component,
and $2^{b_u} > 2^{b_1} + \cdots + 2^{b_{u-1}}$,
the person corresponding to the sink vertex of the component of
size $2^{b_u}$ must be a knight. If the
accusatory components have sizes
$2^{a_1}, \ldots, 2^{a_t}$ and $m = |P|$ then
\[ m = 2^{b_1} + \cdots + 2^{b_u} + 2^{a_1} + \cdots + 2^{a_t}. \]
Hence $t + u \ge B(m)$. The number of questions asked is therefore 
at most $m-B(m)$.
In the usual room of $n$ people known to contain at least~$k$ knights, any set of $2(n-k)+1$
people has a  strict majority of knights. 
Thus, as proved by Aigner in \cite[Theorem~6]{Aigner}, $2(n-k) - B(n-k)$ 
questions suffice to find a knight, even when spies
are unconstrained. 

\section{Proof of Theorem~\ref{thm:liars}}\label{sec:liars}

Throughout this section we suppose that spies always lie.
Let $s = n-k$. By hypothesis
there are at most $s$ spies in the room and $n = q(s+1) + r$, where $0 \le r \le s$.
Each component $C$ of the question graph has a partition $Y$, $Z$,
unique up to the order of the parts, such that the people in $Y$ and $Z$ have opposite identities.
Choosing $Y$ and $Z$ so that $|Y| \ge |Z|$, we define the \emph{weight} of $C$ to be $|Y| - |Z|$.
The multiset of  component weights then encodes exactly the same information as the
 `state vector' in \cite[page 5]{Aigner} or 
the `game position' in \cite[Section~2]{BritnellWildonMajority}, 
\cite[page~384]{SaksWerman} and \hbox{\cite[Section~3]{Wiener}}.
The following lemma also follows from any of these papers, and is proved here only for completeness.

\begin{lemma}\label{lemma:weight}
Suppose that spies always lie. Let $C$ and $C'$ be components in a question graph of weights $c$, $c'$
respectively, where $c \ge c' \ge 1$. Let Persons~$v$ and $v'$ be in the larger parts of the partitions
defining the weights~$c$ and $c'$. Suppose that Person $v$ is asked about Person $v'$, forming a new component
$C \cup C'$. If Person $v$ supports~Person $v'$ then the weight of $C\hskip1pt \cup C'$ is $c + c'$ and 
if Person~$v$ accuses Person~$w$ then the weight of $C \cup D$ is $c-c'$.
\end{lemma}

\begin{proof}
Let $Y$, $Z$ and $Y'$, $Z'$ be the unique partitions of $C$ and $C'$ respectively such that
the people in $Y$ and $Z$ have opposite identities, the people in $Y'$ and $Z'$ have 
opposite identities, and $|Y| - |Z| = c$, $|Y'| - |Z'| = c'$. By assumption $v \in Y$ and $v' \in Y'$.
The unique partition of $C \cup C'$ into people of 
opposite identities is $Y \cup Y'$, $Z\cup Z'$ if
Person $v$ supports Person $v'$, and $Y \cup Z'$, $Z \cup Y'$ if Person $v$ accuses Person $v'$.
The lemma follows.
\end{proof}

\subsection{\sc Lower bounds}\label{subsec:liarslower}
It will be convenient to say that a component in
the question graph is \emph{small} if it contains at most $s$ people.
The identities of people in a small component are ambiguous.

Suppose that it is not known whether a spy is present. 
The Spy Master should answer the first $n-q-1$ questions asked
by the Interrogator with supportive statements. 
After question $n-q-1$ there are
at least $q+1$ components in the question graph, of which at least one  is small.
Hence $\n{T}_L(n,k) \ge n-q$.
Moreover, if $r=0$ then, after question $n-q-1$,
there are at least two small components, say $X$ and~$Y$. If question $n-q$ 
connects~$X$ and~$Y$ then the Spy Master should accuse, otherwise he supports.
In either case at least one more question is required, and so $\n{T}_L(n,k) \ge n-q+1$
when~$r=0$.

The proof is similar if it is known that a spy is present. The Spy Master answers
the first $n-q-2$ questions with supportive statements. This leaves at least $q+2$ components
in the question graph. If $r=0$ or $r=1$ then at least three of these
components are small, and otherwise
at least two are small. The Interrogator is unable to find a spy after $n-q-2$ questions. 
Hence $\c{T}_L(n,k) \ge n-q-1$. Now suppose that $r=0$ or $r=1$. If question $n-q-1$
is between two small components, say $C$ and $C'$, then
the Spy Master should accuse; otherwise he supports.
In the first case it is ambiguous which of $C$ and $C'$ contains spies,
and in the second case there remain two small components and no accusations have been made.
Hence $\c{T}_L(n,k) \ge n-q$ in these cases.

\subsection{\sc Upper bound when $n \not= 2s+1$}\label{subsec:liars}
We start with a questioning strategy which allows
the Interrogator to find a knight 
while keeping the components in the question graph
small.  
See Example~\ref{ex:ex} for an example of the strategy in this context of Theorem~\ref{thm:liarscomb}.

\begin{strategy}[\bf Switching Knight Hunt]
Let $d \in \N$. Let $c_1, \ldots, c_d \in \N$ be 
such that $c_{j+1} > c_1 + \cdots + c_{j}$ for each $j \in \{1,\ldots,d-1\}$.
The starting position for a Switching Knight Hunt is a question graph $G$ 
having distinguished 
components $C_1, C_2, \ldots, C_d$ 
and $C'_2, \ldots, C'_d$ such that
\begin{itemize}
\item[(a)] $C_1$ has weight $c_1$,
\item[(b)] both 
$C_i$ and $C'_i$ have weight $c_i$ for all $i \in \{2,\ldots,d\}$,
\item[(c)] each~$C_i$ has a vertex $p_i$ and each $C_i'$ has
a vertex~$p_i'$, both in the larger part of the partitions defining
the weights of these components.
\end{itemize}
Set $b=1$.

\medskip
\noindent\emph{Step 1.}
If $b = d$ then terminate. Otherwise,
ask Person $p_{b+1}$ about Person~$p_b$, then Person $p_{b+2}$ about Person~$p_{b+1}$, and so on,
stopping either when an accusation is made, or when Person $p_d$ supports Person $p_{d-1}$.
In the latter case the strategy terminates. 
If Person $p_{b+j}$ accuses Person $p_{b+j-1}$
then replace $b$ with $b+j$ and go to Step $1'$.

\medskip
\noindent\emph{Step $\mathit{1'}$.}
This is obtained from Step 1 by swapping $p_i$ for $p_i'$ in all cases, and going to Step $1$
if the strategy does not terminate.
\end{strategy}

\begin{lemma}\label{lemma:switch}
Let $P$ be a set of people in which knights are in a strict majority.
Suppose that $P$ has components $C_1, \ldots, C_d$ and $C_2', \ldots, C_d'$
satisfying the conditions for a Switching Knight Hunt. Let $X = C_1 \cup \cdots \cup C_d$
and $X' = P \backslash X$. Suppose that $X'$ is a union of components of the question graph
and that the components in $X'$ other than
$C_2', \ldots, C_d'$ have total weight at most $c_1-1$.
Let $G$ be the question graph when a Switching Knight Hunt terminates.
If the strategy terminates in Step $1$ then Person $p_d$ is a knight, 
and if the strategy terminates in Step $1'$ then Person $p_d'$ is a knight.
Moreover 
each component in $G$ is either contained in $X$
or contained in $X'$.
\end{lemma}

\begin{proof}
We suppose that the Switching Knight Hunt terminates in Step $1$. (This happens when either
$d=1$, or Person $p_d$ supports Person $p_{d-1}$, or Person $p_d'$ accuses Person $p_{d-1}'$ and there
is a final switch.) The proof in the other
case is symmetric. Suppose that Persons $p_{u_1}$, $p_{u_2'}$, \ldots, $p_{u_{2t-1}}$, $p_{u_{2t}'}$
make accusations, where
$u_1 < u_2' < \ldots < u_{2t-1} < u_{2t}'$. Set \hbox{$u'_0 = 1$}.
After the final question, 
the component of $G$ containing Person $p_{u_i}$ 
is contained
in $X$ and, by Lemma~\ref{lemma:weight}, has weight $c_{u_{i}} - (c_{u_{i}-1}  + \cdots + c_{u_{i-1}'})$.
Similarly
the component of $G$ containing Person $u_i'$ 
is contained in~$X'$ and has
weight $c_{u'_i} - (c_{u'_{i}-1} + \dots + c_{u_{i-1}})$.

Fix $i \in \{1,\ldots, t\}$ and let $u'_{2i-2} = \alpha$, $u_{2i-1} = \beta$ and $u'_{2i} = \gamma$.
The difference between the number of knights
and the number of spies in the components 
$C_{\gamma-1}, \ldots, C_\alpha$ and
$C'_\gamma, \ldots, C'_{\alpha+1}$ of the original question graph is at most
\begin{align*}
& c_{\gamma-1}  + \cdots + c_{\beta+1} + \bigl(c_\beta -  
 ( c_{\beta-1} + \cdots + c_{\alpha+1} + c_\alpha) \bigr) \\
&\hspace*{0.5in} {}+{} \bigl( c_\gamma  - ( c_{\gamma-1}  + \cdots + c_{\beta+1} + c_\beta) \bigr) {}+{} 
 c_{\beta-1} {}+ \cdots + c_{\alpha+1},
\end{align*}
where the top line shows contributions from components in $X$, and the bottom line contributions
from components in $X'$.
This expression simplifies to $c_\gamma - c_\alpha$.  
Hence the difference between
the number of knights and the number of spies in all components of $G$ contained in $P$
\emph{except} for the component containing Person $p_d$, is at most
\[ \sum_{i=1}^t (c_{u'_{2i}} - c_{u'_{2i-2}}) + (c_{d} + \cdots + c_{u_{2t}'+1}) + (c_1-1) = c_{d} + \cdots 
+ c_{u'_{2t+1}} + c_{u'_{2t}} - 1\]
where the second two summands on the left-hand side come from components in $X'$.
The component containing Person $p_d$ has weight $c_d + \cdots + c_{u'_{2t}}$. If the people
in the larger part of this component are spies then spies strictly outnumber knights in $P$,
a contradiction. Hence Person $p_d$ is a knight.
\end{proof}

We remark that in some cases, depending on the structure of the components
$C_i$ and $C_i'$, and provided Persons $p_i$ and $p_i'$ are chosen appropriately,
the Switching Knight Hunt may be effective even when spies are unconstrained.
For example, this is the case in Example~\ref{ex:ex}.

We are now ready to give a questioning strategy that meets the targets set for $\c{T}_L(n,k)$ and
$\n{T}_L(n,k)$ in Theorem~\ref{thm:liars}.
In outline: the Interrogator finds a knight in $K(n,k)$ questions
while also attempting to create $q$ components of size~\hbox{$s+1$} or more,
each with no accusatory edges. 
If he fails in creating these components it is because of an earlier accusation;
asking the knight about the accuser then identifies a spy.
Remarks needed to show that the strategy is well-defined are given in square brackets.

\begin{strategy}[\bf Binary Spy Hunt]
Take a room of $n$ people known to contain at most $s$ spies where $2(s+1) \le n$.
Let $2^{a_1} + 2^{a_2} +  \cdots + 2^{a_d}$ where $d = B(s+1)$ and $a_1 < \ldots < a_d$ be
the binary expansion of $s+1$.

\medskip
\noindent\emph{Phase 1.} Choose disjoint subsets $X$, $X' \subseteq \{1,2,\ldots, n\}$ such
that $|X| = s+1$ and $|X'| = s$. Perform a Binary Knight Hunt in $X$ and then perform a~Binary 
Knight Hunt in $X'$. [The questions asked are consistent with an incomplete
 Binary Knight Hunt in $X \cup X'$.]

\begin{itemize}
\item[(i)]
If an accusation has been made, complete a Binary Knight Hunt
in $X \cup X'$. Then choose any person, say Person $z$, who made an accusation
in Phase $1$, and terminate after asking the knight just found about Person~$z$.

\smallskip
\item[(ii)] If all answers so far have been supportive, go to Phase~2.
\end{itemize}

\medskip
\noindent\emph{Phase 2.} 
\noindent
[The components of the question graph in $X$
have sizes $2^{a_1}, \ldots, 2^{a_d}$. 
Since $s = (1 + \cdots + 2^{a_1-1}) + 2^{a_2} +
\cdots + 2^{a_d}$, there are components in $X'$ of sizes $2^{a_2}, \ldots, 2^{a_d}$.
No accusations have been made so far, hence the size of each component is equal to its weight.]
Perform a Switching Knight Hunt in $X \cup X'$ and go to~Phase~3.

\medskip
\noindent\emph{Phase 3.} 
[By Lemma~\ref{lemma:switch}, each component of the question graph
is either a singleton, or contained in $X$ or contained in~$X'$. There
are $(q-2)(s+1) + r+1$ singleton components not contained in $X \cup X'$.]
Let Person $w$ be the knight found at the end of~Phase~2.
Ask questions  to create $q-1$ components of
size $s+1$, one component of size $s$, and $r+1$ singleton components.
At the first accusation, stop building components and 
ask Person $w$ about the person who made the accusation. Then terminate.
(Thus Phase~3 ends after one question
if there was an accusation in Phase~2.) If no accusations are made,
go to Phase~4.

\medskip
\noindent\emph{Phase 4.} Let Persons $x_1, \ldots, x_{r+1}$ be in the
singleton components of the graph. Let Person $y$ be in the component
of size $s$.

\begin{itemize}
\item[$\bullet$] If $r=0$ then ask Person $w$ about Person $x_1$. 
If he accuses, Person~$x_1$ is a spy. If he supports, and a spy
is known to be present, then Person $y$ is a spy. Otherwise
asking Person $w$ about Person $y$ either shows that Person $y$ is a spy,
or proves that no spies are present.

\smallskip
\item[$\bullet$] If $r \ge 1$ then ask Person $y$ about Person $x_{r+1}$.
If he accuses then asking Person $w$ about Person $y$ identifies a spy. Otherwise
ask Person $w$ about Persons $x_1$, \ldots, $x_{r-1}$. Any accusation
identifies a spy. Suppose that all these people turn out to be knights.
If a spy is known to be present, then Person $x_{r}$ is a spy;
otherwise asking Person $w$ about Person $x_{r}$  either shows that Person $x_{r}$ is a spy,
or proves that no spies are present.
\end{itemize}

\end{strategy}

In Example~\ref{ex:ex} the Binary Spy Hunt is shown ending in Phase~4.

\begin{lemma}\label{lemma:BSHworks}
Let $s = n-k$ and suppose that $n \ge 2(s+1)$.
Assume that spies always lie. Suppose that a
Binary Spy Hunt is performed in the room of $n$ people. If a spy is known to be present, then
a spy is found after at most $\c{T}_L(n,k)$ questions. Otherwise,
after $\n{T}_L(n,k)$ questions, either a spy is found, or it is clear that no spy is present.
\end{lemma}

\begin{proof}
At the beginning of Phase~4 the question graph has $q+r+1$ components. Hence if
an accusation is made in an earlier phase then, after the accusation, the
question graph has at least $q+r+1$ components. Therefore, after Person $w$ is used to identify the
accuser, there are at least $q+r$ components. This question identifies
either the accuser or the person accused as a spy, and so a spy is found
after at most $n-q-r$ questions. This meets all the targets in Theorem~\ref{thm:liars}.

Suppose the strategy enters Phase~4. If a spy is known to be present then
the Interrogator asks $1$ question if $r=0$, at most $2$ questions if $r=1$, and at most $r$ questions if $r \ge 2$.
The final numbers of components are at least $q$, $q$ and $q+1$, respectively.
If a spy is not known to be present then
the Interrogator asks at most $2$ questions if $r=0$, exactly $2$ questions if $r=1$, 
and at most $r+1$ questions if $r \ge 1$.
The final numbers of components are at least $q-1$, $q$ and $q$, respectively.
This meets the targets in Theorem~\ref{thm:liars}.
\end{proof}

This completes the proof of Theorem~\ref{thm:liars} in the case
$n\not= 2s+1$. For later use in the proof of Theorem~\ref{thm:liarscomb} in \S\ref{sec:liarscomb}
we record the following result on the Binary Spy Hunt.

\begin{proposition}\label{prop:BSHknight}
Suppose that spies always lie and that a Binary Spy Hunt enters Phase~2. A knight
is found at the end of Phase~2 after exactly $K(n,k)$ questions.
\end{proposition}

\begin{proof}
Let $s = n-k$ and
suppose as before that 
$s+1 = 2^{a_1} + 2^{a_2} + \cdots + 2^{a_d}$
where $d = B(s+1)$ and $a_1 < \ldots < a_d$. Let~$X$ and~$X'$ be the subsets of size $s+1$ and~$s$,
respectively, chosen in the strategy.
After the Switching Knight Hunt ends Phase 2, in the 
quotient of the question graph obtained by identifying Persons~$p_i$ and~$p_i'$ for $i \in \{2,\ldots, d\}$,
the images of the 
vertices $p_d, \ldots, p_1$ form a new directed path of length $d$. Hence exactly $d-1$ questions are asked
in Phase~2. The number of components after Phase~1 in~$X$ and~$X'$ are~$d$ and $d-1+a_1$, respectively;
thus after Phase~2, the number of components in $X \cup X'$ is $d+a_1$. Hence the number of questions
asked in Phases~1 and~2~is
\[ 2s+1 - (d+a_1) = 2s - B(s) \]
which equals $K(n,k)$, as required.
\end{proof}

\subsection{\sc Upper bound when $n=2s+1$}\label{subsec:liarsexcep}
The remaining case when $n=2s+1$  has a number of exceptional features.
When $n=3$, it is clear that a single question cannot identify a spy, while any two
distinct questions will, so 
$\n{T}_L(3,2) = \n{T}_L(3,2) = 2$, as required.
When $n=5$ and $s=2$, the Spy Master should support on his first answer.
He may then choose his remaining answers so
that the (undirected) question graph after three questions appears in Figure~1 below.
In each case a spy must be present, it is consistent that the spies
lied in every answer, and no spy can be identified without asking one more question.
Hence $\c{T}_L(5,3) = 4$ and $\n{T}_L(5,3) = 4$.

\begin{figure}[h!]
\includegraphics{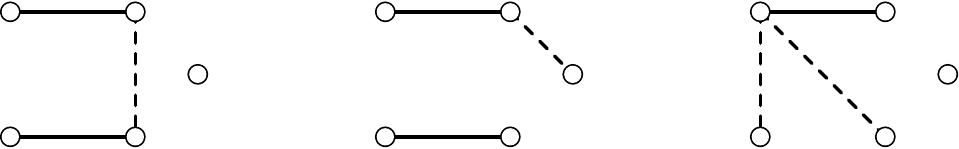}
\caption{\small\setlength{\baselineskip}{12pt} Undirected question graphs after four questions when $n=5$ and $k=3$
with
optimal play by the Spy Master.} 
\end{figure}

Now suppose that $s \ge 3$. The lower bound  proved in \S\ref{subsec:liarslower} shows
that $\c{T}_L(2s+1,s+1) \ge 2s-1$ and $\n{T}_L(2s+1,s+1) \ge 2s$.

We saw in \S 2 that a Binary Knight Hunt will find 
a knight, say Person~$w$, 
after at most $2s-B(s)$ questions. At this point the question graph is a forest.
If $s$ is not a power of two then, since $B(s) \ge 2$, the Interrogator
can  
ask Person $w$ further questions until exactly $2s-2$ questions have been asked,
choosing questions so that the
question graph remains a forest.
Suppose that after question $2s-2$ the components
in the question graph are $X$, $Y$ and $Z$,
where $X$ is the component containing Person~$w$. 
If an accusation has been made by someone in $X$, 
then a spy is known.
Moreover, if an accusation has been made by someone in $Y$ or $Z$, then 
asking Person $w$ about this person will identify a spy.
Suppose 
that no accusations have been made. 
If it is known that a spy is present then Person $w$ will support
a person in component $Y$ if and only if everyone in component $Z$ is a spy,
and so one further question suffices to find a spy.
Otherwise, two questions asked to Person $w$ about people in components $Y$ and $Z$
will find all identities. 

The 
remaining case is when $s = 2^e$ where $e \ge 2$.
It now requires $2s-B(s) = 2s-1$ questions to find a knight using
a Binary Knight Hunt. 
One further question will connect the two remaining
components in the question graph, finding all identities in $\n{T}_L(2s+1,s+1) = 2s$ questions.
Suppose now that a spy is known to be present. Then the danger is that,
as in the question graphs shown in Figure~1, 
after asking the target number of $2s-1$ questions, the Interrogator
succeeds in identifying a knight, but not a spy.
When $s=4$ this trap may be avoided using
the following lemma.

\begin{lemma} $\c{T}_L(9,5) \le 7$.
\end{lemma}

\begin{proof}
The table in Figure~2 shows
the sequence of questions the Interrogator should ask, together with an
optimal sequence of replies from the Spy Master. The final column 
gives the continuation if the Spy Master gives the opposite
answer to the one anticipated in the main line. (The further questions
in these cases are left to the reader.)
It is routine to check that in every case
the Interrogator finds a spy after at most  seven questions.
\end{proof}
%
\newlength{\expected}
\settowidth{\expected}{\mbox{Anticipated}}
\newlength{\opposite}
\settowidth{\opposite}{\mbox{Continuation for}}
\addtolength{\opposite}{-7pt}

\begin{figure}[h!]
\newcommand{\No}{Support}
\newcommand{\Yes}{Accuse} 
\small
\centerline{\begin{tabular}{llll}
\toprule
Components of question graph & Question & 
\parbox{\expected}{\setlength{\baselineskip}{12pt}\raggedright Anticipated\ answer}
  & \parbox{\opposite}{\setlength{\baselineskip}{12pt}\raggedright Continuation for opposite answer}  \\ \midrule
$\{1\},\, \{2\},\, \{3\},\, \{4\},\, \{5\},\, \{6\},\, \{7\},\, \{8\},\, \{9\}$ & $(1,2)$ & \No${}^\star$ & $(3,4)$ BKH \\
$\{1,2\},\, \{3\},\, \{4\},\, \{5\},\, \{6\},\, \{7\},\, \{8\},\, \{9\}$        & 
$(1,3)$ 
& \No & $(4,5)$ BKH \\
$\{1,2,3\},\, \{4\},\, \{5\},\, \{6\},\, \{7\},\, \{8\},\, \{9\}$               & $(4,5)$ & \No & 
$(1,6)$ \\ %
$\{1,2,3\},\, \{4,5\},\, \{6\}, \{7\},\, \{8\},\, \{9\}$                        & 
$(4,6)$ 
& \Yes & 
$(1,4)$ \\
$\{1,2,3\},\, \{4,5\} \mid \{6\}, \{7\},\, \{8\},\, \{9\}$                      & 
$(4,7)$ 
& \No${}^\star$ & $(1,5)$ \\
$\{1,2,3\},\,  \{4,5,7\} \mid \{ 6\},\, \{8\},\, \{9\}$                         & 
$(1,8)$ 
& \Yes 
& $(1,4)$  \\
$\{1,2,3\} \mid \{8\},\,  \{4,5,7\} \mid \{ 6\},\,\{9\}$                        & 
$(1,9)$ 
& 
\multicolumn{2}{l}{\raisebox{-6.0pt}{\parbox{1.6in}{\raggedright\setlength{\baselineskip}{12pt} \No: Person $8$ is a spy \Yes: Person $6$ is a spy}}} \\ \bottomrule
\end{tabular}}
\caption{\small\setlength{\baselineskip}{12pt} In a nine person room, seven questions suffice to find a spy.
The question `Person $x$, is
Person $y$ a spy?' is shown by $(x,y)$.
Components of the question graph known to contain a spy are shown by $X \mid Y$ 
where the people in $X$ and $Y$ have opposite identities.
Answers marked $\star$ are the unique optimal replies by the
Spy Master. The abbreviation BKH indicates that the continuation
is a Binary Knight Hunt. (If the second question results in an accusation,
regard the component $\{1,2,3\}$ of weight $1$
as the singleton $\{2\}$.) 
}
\end{figure}

Now suppose that $s = 2^e$ where $e \ge 3$. 
Let 
\[ \{1,2,\ldots, 2s+1 \} = X_1 \cup X_2 \cup \cdots \cup X_8 \cup \{2s+1\} \]
where the union is disjoint and $|X_i| = 2^{e-2}$ for all $i$.
The Interrogator should start by performing a separate Binary Knight Hunt
in each $X_i$. 
Suppose that an accusation is made, say when two components both of size $2^f$ 
are connected. Let $Y$ be the set of people not in either of these components.
The questions asked so far in $Y$ form an incomplete Binary Knight Hunt in~$Y$.
Since $|Y| = 2(2^e-2^f) + 1$, a knight may be found after
\[ 2(2^e-2^f) - B(2^e-2^f) \]
further questions. Asking this knight about
a person in the accusatory component
of size $2^{f+1}$ identifies a spy. The total number of questions asked is
$2^{f+1} - 1 + 2(2^e-2^f) - B(2^e-2^f) + 1 = 2^{e+1} - B(2^e-2^f) = 2s - B(2^e-2^f)$.
Since $f \le e-3$,  this is strictly less than $2s-1$.

If no accusations are made then, after the eight Binary Knight Hunts are performed, each~$X_i$
is a connected component of the question graph
containing $2^{e-2}$
people of the same identity. There is also a
final singleton component containing Person $2s+1$.
Let Person~$p_i$ belong to $X_i$ for each~$i$, and let
$p_9 = 2s+1$. It is  routine to check that replacing $i$ with $p_i$ 
in the question strategy shown in Figure~2 will now find a spy
in at most $7$ more questions, leaving a final question graph with two components.

\section{Proof of Theorem~\ref{thm:knights}}\label{sec:knights}

Suppose that spies always lie and
that $G$ is a question graph having components in which
the Interrogator can correctly claim that Person $x$ is a spy. Let $C$ be the component containing Person $x$ and
let $X$, $Y$ be the partition of $C$ into people of different types,
chosen so that $x \in X$. If $|X| \ge |Y|$ then, given any assignment of identities to the people
in the room that makes the people in $X$ spies, we can switch knights and spies in component $C$
to get a new consistent assignment of identities. 
Hence $|X| \le |Y|$ and 
the people in~$Y$ must be knights.
Thus $E_L(n,k) = K_L(n,k)$. Since
\[ E_L(n,k) \le E_S(n,k) \le K_S(n,k) \]
is obvious and $K_S(n,k) = K_L(n,k)$ was seen in the introduction, it follows
that $E_S(n,k) = K_S(n,k) = E_L(n,k) = K_L(n,k)$.
Since $K(n,k) = 2(n-k) - B(n-k) \le n-2$, there is a person not involved 
in any question by question $K(n,k)$. Therefore the same result holds
for the corresponding quantities defined on the assumption that a spy is present.

For the next part of Theorem~\ref{thm:knights} we must recall a basic result on the
reduction of the majority game to multisets of weights. 

\begin{lemma}\label{lemma:weightClaim}
Suppose that a question graph 
has components $C_1, \ldots, C_d$. Let~$c_i$ be the weight of $C_i$ and let $c_1 + \cdots + c_d = 2s+e$
where $e = k - (n-k)$ and $s \in \N_0$. The
identities of the people in component $C_i$ are unambiguous if and only if $c_i \ge s+1$. 
\end{lemma}

\begin{proof}
See \cite[Equation (14)]{Aigner} or \cite[Section 2]{BritnellWildonMajority}.
\end{proof}

Let $t = K(n,k)$. It is clear that $N_L(n,k) \le N_S(n,k) \le K_S(n,k)+1$, so 
to show that $N_L(n,k) = N_S(n,k) = t+1$, it suffices to show that
$N_L(n,k) \ge t + 1$. 
The Spy Master can ensure that after $t-1$ questions the Interrogator is unable to identify
a knight.
Suppose one of the first $t$ questions forms a cycle in the question graph.
By the remarks on the majority game following the statement of Theorem~\ref{thm:knights},
this question is redundant from the point of view of finding a knight. The
results already proved in this section show that no identities can be found until
a knight is found. We may therefore assume that
the question graph after question $t$ is a forest.

Let the sum of the weights of the components of the question graph after question $t-1$
be $2s+e$ where $e = k - (n-k)$. By Lemma~\ref{lemma:weightClaim}
each component has weight at most $s$.
Suppose that on question $t$ the Interrogator asks a person in component~$C$ 
about a person in component $C'$. Let $c$ be the weight of $C$ and let $c'$ be 
the weight of $C'$. By the reduction to the majority game, we
may assume that $c \ge c'$. We consider two cases.

\begin{itemize}
\item[(i)]
If $1 \in C \cup C'$ then the Spy Master supports. The
weight of the component containing Person 1 is unchanged, so by Lemma~\ref{lemma:weightClaim},
the identity of Person 1 is still ambiguous.

\item[(ii)]
If $1 \not\in C\hskip1.75pt \cup\hskip1.75pt C'$ then the Spy Master accuses. 
By Lemma~\ref{lemma:weight}, the weight of
the new component containing Person 1 is $c - c'$. The sum of all component weights is now
$2(s-c') - e$, and we have $c - c' \le s-c'$. 
By Lemma~\ref{lemma:weightClaim} the identity of Person 1 is still ambiguous.
\end{itemize}

Hence $N_L(n,k) = N_S(n,k) = t$, as required.
In case (ii) a spy is clearly present. In case (i), the
question graph after question $t$ has a component $C \cup C'$ not
containing Person $1$. If spies are unconstrained then a source
vertex in this component may be a spy. 
 Hence $\c{N}_S(n,k) = t+1$. 
Moreover, unless $n = 2^{e}+1$ and $k = 2^{e-1} +1$ for some $e \in \N$ we have
$t +1 \le n-2$, and so after question $t+1$ there is person not yet 
involved in any question, implying that
 $\c{N}_L(n,k) = t+1$. 
The proof of Theorem~\ref{thm:knights}
is completed by the following lemma which deals with the exceptional case when $t = n-2$.

\begin{lemma}\label{lemma:knightsexcep}
If $n = 2^{e+1}+1$ and $k = 2^{e} +1$ then $\c{N}_L(n,k) = n-2$.
\end{lemma}

\begin{proof}
The Interrogator performs a Binary Knight Hunt using Persons~$2$ up to $n$.
If there is no accusation on or before question $n-2$ then Person $1$ is a spy. 
Suppose that the first accusation occurs when two components of size $2^f$ are connected.
If $f = e$ then Person $1$ is identified as a knight after $n-2$ questions. Otherwise, ignoring
the new component of weight $0$, the new multiset of component
weights is consistent with a Binary Knight Hunt in a room of $2^{e+1} - 2^{f+1} +1$ people,
known to contain at least $2^e - 2^f + 1$ knights.
A knight may therefore be found after at most 
\[ (2^{f+1} - 1) + 2(2^{e} - 2^{f}) - B(2^e-2^f) = 2^{e+1} - B(2^e-2^f) - 1 \le n- 3 \]
questions, and Person $1$'s identity found by question $n-2$.
\end{proof}

\section{Proof of Theorem~\ref{thm:liarscomb}}\label{sec:liarscomb}

Let $s = n-k$. We must deal
with the cases $n \ge 2(s+1)$ and $n = 2s+1$ separately.

\begin{proof}[Proof when $n \ge 2(s+1)$]
Let $2^a$ be the largest power of two such that $2^a \le s+1$.
Perform a Binary Spy Hunt, as described in \S\ref{subsec:liars},
choosing the sets $X$ and $X'$ of sizes $s+1$ and $s$, respectively,
so that $1 \in X$. Whenever permitted in the Binary
Knight Hunt in Phase 1, ask questions to Person $1$, or failing that, within $X$.
Suppose there is an accusation in Phase~1; then
the Binary Spy Hunt is completed in $X \cup X'$ and 
a knight, say Person~$w$, is found after at most $K(n,k)$ questions. 
If the first accusation is in $X$ then an easy inductive argument shows
that after question $K(n,k)$ either Person $1$ is in an accusatory component, or in
the same component as Person~$w$.
If the first accusation is in~$X'$ then, before this
accusation, Person $1$ is in a non-accusatory
component in~$X$ of size $2^a$; since at most one other component
of this size can be formed in~$X'$, the same conclusion holds.
If Person $1$ is in
an accusatory component after question $K(n,k)$ then 
asking Person~$w$ about Person $1$ 
in question $K(n,k)+1$
both determines the identity of Person $1$ and finds a spy; 
in the other case case Person $1$'s identity is known,
and question $K(n,k)+1$ may be used to find a spy.
Let $n = q(s+1) + r$ and note that $K(n,k) + 1 = 2s-B(s) + 1$. If $r \le 1$ then 
\[ \c{T}_L(n,k) = n-q = qs+r \ge 2s-B(s) + 1 \]
with equality if and only if $q=2$, $r=0$ and $B(s) = 1$. If $r \ge 2$ then
\[ \c{T}_L(n,k) = n-q-1 = qs+r-1 > 2s-B(s) + 1.\]
Thus the targets for finding a spy are met.

Now suppose the Binary Spy Hunt enters
Phase~2. By Proposition~\ref{prop:BSHknight}
a knight, say Person $w$, is found at the end of Phase~2 after $K(n,k)$ questions.
If after Phase~2, Persons~$1$
and~$w$ are in the same component of the question graph, then the identity of Person 1 is known, and 
the strategy continues as usual, either finding a spy or proving that no spy is present.
If they are in different components then at least one switch from $X$ to $X'$ occurred
in the Switching Knight Hunt 
in Phase 2, and so there is an accusatory edge
in the component of Person~1, and the earlier argument applies.
\end{proof}

\begin{proof}[Proof when $n = 2s+1$] 
The case $s=1$ is easily dealt with. If $s > 1$ then $K(2s+1,s+1) = 2s-B(s)$,
$\c{T}_L(2s+1,s+1) = 2s-1$ and $\n{T}_L(2s+1,s+1) = 2s$.
If $B(s) \ge 2$ then perform a Binary Knight Hunt, always asking questions to Person $1$ whenever
permitted. This finds a knight, say Person $w$, in at most
$2s-2$ questions.
Again either Person~$1$ is in a component with an accusatory edge,
or in the same component as Person $w$. The former case is as earlier. 
In the latter case asking further
questions to Person~$w$, while keeping the question graph a forest, meets 
the targets for finding a spy.

Suppose that $s$ is a power of two. If $s \ge 4$ 
then $K(2s+1,s+1) = 2s-1$. As shown
in \S\ref{subsec:liarsexcep}, it is possible to find a knight by question $2s-1$
and all identities by question $2s$. Suppose that a spy is known to be present.
Then the strategy in \S\ref{subsec:liarsexcep} finds both a knight and a spy 
by question $2s-1$. This leaves one further question to determine
the identity of Person 1. 
The case $s=1$, is easily dealt with,
as is the case $s=2$ when $\c{T}_L(5,3) = 4$.
\end{proof}

In all cases, after the
question when a spy is identified, or, in the case of $\n{T}_L(n,k)$, when it is clear that
no spies are present, the question graph is a forest.
Hence all identities may be obtained in $n-1$ questions, as claimed in the introduction.

\enlargethispage{30pt}
\begin{example}\label{ex:ex}
Figure~3 below shows an example of the Binary Spy Hunt used to prove Theorem~\ref{thm:liarscomb}
in which all four phases of the strategy are required.
\end{example}

\begin{figure}[h]
\begin{center}
\vspace*{-3pt}
\scalebox{1}{\includegraphics{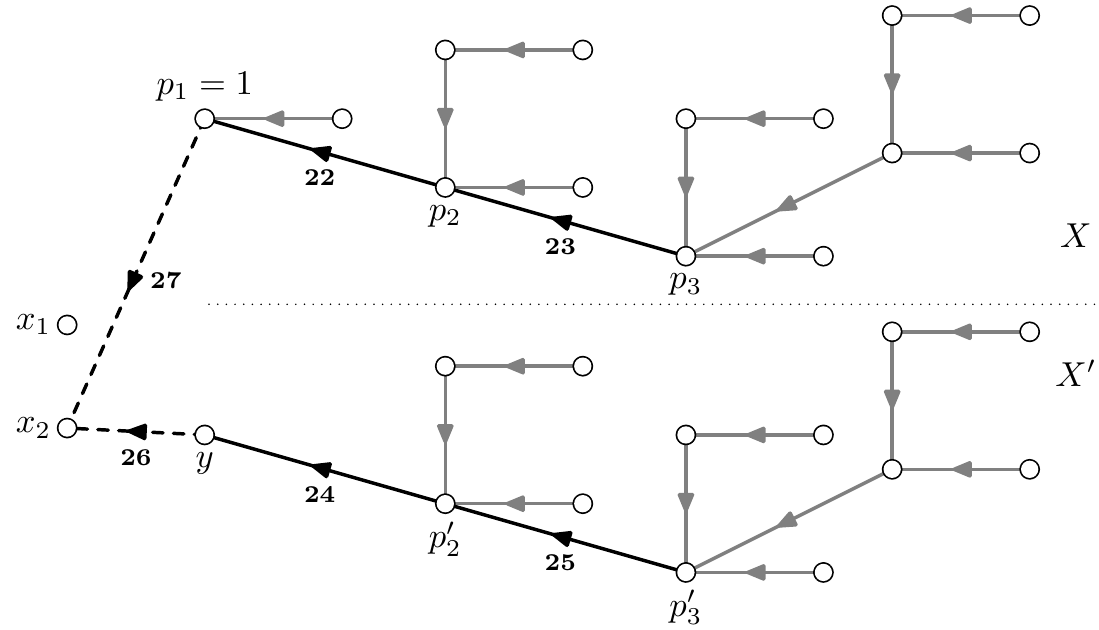}}
\end{center}

\vspace*{-9pt}
\caption[]{Example of the Binary Spy Hunt used to prove Theorem~\ref{thm:liarscomb}. 
We take $s = 13$ and $n = 29$. 
Notation
is as in the description of this strategy in \S\ref{sec:liars}. 
The set $X$ consists of the $14$ vertices above the dotted line,
and $X'$ consists of the $13$ vertices other than $x_1$ and $x_2$ below the dotted line.
Phase~1 lasts $21$ questions (indicated by grey arrows)
and leaves components in $X$ of sizes $2$, $4$ and $8$ and
components in $X'$ of sizes $1$, $4$ and $8$. Later questions are numbered. 
Phase~2 ends after question $K(29,13) = 23$ with Person $1$ identified as a knight.
Phase~3 ends after question $25$. Person $x_2$
is identified as a spy at the end of Phase~4
after $\c{T}_L(29,13) = 27$ questions. One further question will
find the identity of Person $x_1$, determining all identities in $n-1 = 28$ questions.

\medskip
To give a more interesting example of the Switching Knight Hunt (used in Phase~2)
we remark that
if $p_2$ had accused $p_1$ on question $22$ then the strategy would have switched
to $X'$. Suppose that $p_3'$ then supports~$p_2'$. Then $p_2'$ is identified as a knight
after $23$ questions, and asking $p_2'$ about $1$ both finds a spy and determines
the identity of Person~$1$.
}

\end{figure}

\section{Proof of Theorem~\ref{thm:spies}}\label{sec:spies}

We now turn to the first of the two main theorems dealing with unconstrained spies.
Suppose it is known  that a spy is present. After $n-2$ questions have been
asked, there are two people in the room who have never
been asked about. If no accusations have been made then
 it is consistent that exactly
one of these people is a spy. Hence $\c{T}_S(n,k) \ge n-1$. A similar
argument shows that $\n{T}_S(n,k) \ge n$.


To establish the upper bounds in Theorem~\ref{thm:spies} we use
a modified version of the questioning strategy used in \cite{Blecher}
and \cite{WildonKS} 
to find everyone's identity. 
We show in \S\ref{sec:spiescomb} below that a suitable modification
of the Binary Knight Hunt can also be used, provided $k-1$ is not a power of two.
It is worth noting that the unmodified Binary Knight Hunt is ineffective, since 
it takes $2^{a-1}$ questions to rule out the presence
of spies in a component of size $2^a$ whose sink vertex is a knight.

\begin{strategy}[\bf Extended Spider Interrogation Strategy]{\ }

\smallskip
\noindent\emph{Phase 1.} 
Ask Person $1$ about Person $2$, then Person
$2$ about Person $3$, until either there is an accusation, or Person $n-1$ supports Person~$n$.
If there is an accusation, say when Person $p$ accuses Person $p+1$, go to Phase~2,
treating Person $p$ as a candidate who has been supported by $p-1$ people and 
accused by one person.
Otherwise terminate.

\medskip
\noindent\emph{Phase 2.} Set $\ell = n-k$.
Continue to ask further people about the chosen candidate
until either
\begin{itemize}
\item[(a)] strictly more people have accused the candidate
than have supported him, \emph{or} 
\item[(b)] at least $\ell$ people have supported
the candidate.
\end{itemize} 
If Phase~2 ends in~(a) then replace $\ell$ with $\ell-m$, where $m$ is the number
of people accusing the candidate, and repeat
Phase~2, choosing as a new candidate someone who has not yet been involved in proceedings.
If Phase~2 ends in (b) then the strategy terminates. 
(Thus Phase~2 ends immediately if and only if $p > \ell$.)
\end{strategy}

If the strategy terminates in Phase~1 then $n-1$ questions have been asked.
If a spy is known to be present then Person $1$ is a spy; otherwise asking Person $n$
about Person $1$ will decide whether any spies are present, using $n$ questions in total.

Whenever a candidate is discarded in Phase~2, the connected component of
the question graph containing him contains at least as many spies as knights.
This shows that it is always possible to pick a new candidate when required by Phase~2,
and that the candidate when the strategy terminates is a knight. Let this knight be Person $w$.
If Person $w$ has been accused by anyone then a spy is known. Otherwise asking Person $w$
about Person~$p$ from Phase~1 finds a spy in one more question. In either case the question
graph remains a forest, and so at most $n-1$ questions are asked.

This shows that $\n{T}_S(n,k) \le n$ and $\c{T}_S(n,k) \le n-1$, completing the proof
of Theorem~\ref{thm:spies}.

\section{Proof of Theorem~\ref{thm:spiescomb}}\label{sec:spiescomb}

We first deal with the case $n=7$ and $k=4$ since this shows the  obstacle addressed
by the questioning strategy used in the main part of the proof. For a conditional
generalization of the following lemma, see Corollary~\ref{cor:cor}.

\begin{lemma}\label{lemma:spiescomb74}
Let $n= 7$ and $k=4$. Suppose that spies are unconstrained and that a spy is known to be present.
There is no questioning strategy that will both find a knight
by question $K(7,4) = 4$ and find a spy by question $\c{T}(7,4) = 6$.
\end{lemma}

\begin{proof}
It is easily shown that, even if spies always lie, the Interrogator has only two questioning sequences that find a knight
by question $4$ against best play by the Spy Master.
Representing positions by multisets of weights with multiplicities indicated by exponents,
they are
\begin{align*}
\{1^7\} \rightarrow \{2,1^5\} \rightarrow \{2^2,1^3\} &\rightarrow \{1^3,0\} \rightarrow \{2,1,0\} \\
\ldots & \rightarrow \{2^3,1\} \rightarrow \{4,2,1\}.
\end{align*}
In either case the Interrogator must ask a question that connects two components of size $2$. 
If the Interrogator's question creates an edge into a source vertex, the Spy Master should accuse.
The Interrogator is then unable to find a knight by question $4$. If the new edge is into a sink
vertex, the Spy Master should support. This may reveal
a Knight by question $3$, but the Interrogator is then unable to find a spy by question $6$.
\end{proof}

A similar argument shows that 
there is no questioning strategy that will both
find a knight by question $K(7,4) = 4$ and either find a spy or prove that no spies
are present by question $\n{T}(7,4) = 7$.

For the main part of
 Theorem~\ref{thm:spiescomb} we need the following strategy which finds a knight after $K(n,k)+1$ questions.
The comment in square brackets shows that the strategy is well-defined. 

\begin{strategy}[\bf Modified Binary Knight Hunt]
Let $P$ be a subset of $2s+1$ people in which at most $s$ spies are present
and let $X$ be a subset of $P$ of size $2^{a+1}$, where 
 $a \in \N_0$ is greatest such that $2^a \le s$.

\medskip
\noindent\emph{Phase 1.} [Each component in $X$ is a directed path. 
There are distinct components in $X$ of equal size
unless $X$ is connected.] If $X$ is connected then terminate.
Otherwise choose two components $C$ and $C'$ in $X$ of equal size and
ask the sink vertex in $C$ about the source vertex in $C'$. If there
is an accusation go to Phase~2, otherwise continue in Phase~1.

\medskip
\noindent\emph{Phase 2.} 
Disregard the accusation ending Phase~1
and complete a Binary Knight in $P$, now connecting sink vertices as usual.
\end{strategy}

\begin{proof}[Proof of Theorem~\ref{thm:spiescomb}]
Let $s = n -k$.
Choose
a subset $P$ of $2s+1$ people such that $1 \in P$ and choose the subset $X$ of $P$ so 
that $1 \not\in X$. Perform a Modified Binary Knight Hunt in $P$.

Suppose this ends in Phase~1 after $2^{a+1}-1$ questions. Since
$2^{a+1} > s$ the sink vertex in $X$ is a knight.
Since $s \ge 2^a + B(s) - 1$ we have $2^{a+1}-1 \le 2s+1-2B(s) \le 2s - B(s) = K(n,k)$.
Let Person $w$ be the knight just found.
Ask Person $w$ about Person $1$, and then about each of the other people in singleton
components of the question graph. If there are no accusations after $n-1$
questions then either no spies are present, or the unique source vertex
in the component of Person $w$ is a spy: if necessary this can be decided in one more question.

Suppose the Modified Binary Knight Hunt ends in Phase~2 after $K(n,k)+1$ questions. Ask
Person $w$ about Person $1$, and then about the first person to make an accusation.
This identifies a spy in at most $K(n,k)+2 = 2s-B(s)+2$
questions.
Unless $n = 2s+1$ and $s$ is a power of two, this meets the target $\c{T}_S(n,k) = n-1$,
and in any case meets the target $\n{T}_S(n,k) = n$. 

In the exceptional case when $s$ is a power of two
and $n = 2s+1$, the target for finding a knight
is $K(2s+1,s+1)+1 = 2s-B(s)+1 = 2s$. To motivate Problem~\ref{problem:spiesboth}
we prove a slightly stronger result in this case, 
using the Extended Spider Interrogation
Strategy from \S\ref{sec:spies}. 
If the strategy is in Phase~1 after $2s-1$ questions then a knight is known.
Suppose the strategy enters Phase~2.
The first candidate can be rejected no later than question $2s-1$ (after being
supported by $s-1$ people and accused by $s$ people), so if the first candidate
is accepted then a knight is known by question $2s-1$. If the first candidate
is rejected then a knight is found when the question graph has at least two components,
so by question $2s-1$ at the latest. 
 In all cases a spy is found by question $\c{T}_S(2s+1,s+1) = 2s$.
\end{proof}

\section{Further results and open problems}

\subsection{Finding all identities}\label{subsec:all}
A related searching problem asks for the identity of 
every person in the room. 
Blecher proved in~\cite{Blecher} that, for unconstrained spies, $n+(n-k)-1$
questions are necessary and sufficient to find everyone's identity.  
This result was 
proved independently by the author in \cite{WildonKS} using similar arguments.
It follows from~\hbox{\cite[\S 3.3]{WildonKS}} that~$n+(n-k)-1$ 
questions may be required even if 
spies lie in every answer (but cannot be assumed to do so),
and the first question is answered with an accusation, thereby 
guaranteeing that at least one spy is present. 

Let $A(n,k)$ be the number of questions necessary
and sufficient to determine all identities when spies always lie.
Let $n = q(n-k+1) + r$ where $0 \le r \le n-k$, as in Theorem~\ref{thm:spies}. Aigner proved
in \cite[Theorem~4]{Aigner} that
\[ A(n,k) = \begin{cases} n-q+1 & \text{if $0 \le r \le 1$} \\
n-q+\epsilon_{(n,k)} & \text{if $2 \le r < n-k$} \\ 
n-q & \text{if $r = n-k$.} \end{cases} \]
where $\epsilon_{(n,k)} \in \{0,1\}$.
It is notable that if $r=0$ or $r=n-k$ then $A(n,k) = \n{T}_L(n,k)$,
and so it is no harder to find a spy or to prove that everyone in the room is a knight
than it is to find all identities in these cases. When $r=1$ we have $A(n,k) = \n{T}_L(n,q) + 1$.
Inspection of the proof of Theorem~4 in \cite{Aigner}
shows that Aigner's result  holds unchanged if it is known that a spy is present.

Aigner makes the plausible suggestion that $A(n,k) = n-q+1$ whenever $r < n-k$.
However this is not the case.
In fact, if $n \le 30$ and $r < n-k$ then
$A(n,k) = n-q$ if and only if
\[ (n,k) \in \left\{ \begin{matrix}
(13,9), (16,11),(18,14),  (19,13), (21,14), (22,15), (22,17), \\
(23,19), (24,16), (25,17),  (25,19), (26,17), (26,20), (27,18),   \\ 
(28,19), (28,23), (28,24), (29,19), (29,22), (30,20),  (30,23)\end{matrix} \right\}. \]
This can be checked by an exhaustive search of the game tree, using the 
program \texttt{AllMajorityGame.hs} available from the
author's website\footnote{See \url{www.ma.rhul.ac.uk/~uvah099/}}.
The following problem therefore appears to be unexpectedly deep.

\begin{problem}
Determine $A(n,k)$ when $n = q(n-k+1) + r$ and $2 \le r < n-k$.
\end{problem}

\subsection{The majority game}
The values of $K(n,k)$ were found for all $n$ and~$k$ in \cite{BritnellWildonMajority},
but many natural questions about the majority game remain open. 
Given a 
multiset $M$ of component weights and $e \in \N$ such that the sum
of the weights in $M$ has the same parity as $e$, let $n-V_e(M)$ be the
minimum number of questions that are necessary and sufficient to find a knight
starting from the position $M$,
when spies always lie and
 the excess of knights over spies is at least~$e$.
Thus $V_e(M)$ is the number of components in the final position, assuming optimal play.

\begin{problem}\label{problem:majority_values}
Give an algorithm that computes $V_e(M)$ that is qualitatively faster than searching the game tree.
\end{problem}

For multisets $M$ all of whose elements are powers of two, the Binary Knight
Hunt gives a lower bound on $V_e(M)$. The Switching Knight Hunt gives a lower
bound in some of the remaining cases.
In \cite{BritnellWildonMajority} a family of statistics $SW_e(M)$ were defined, generalizing
the statistic $\Phi(M) = SW_1(M)$ used in \cite{SaksWerman}. 
In \cite[Section 5]{BritnellWildonMajority} it was shown that $V_e(M) \le SW_e(M)$. However,
Lemma~7 in \cite{BritnellWildonMajority} shows that the difference may be arbitrarily large.
It therefore seems that fundamentally new ideas will be needed for Problem~\ref{problem:majority_values}.

One natural special case occurs when $e=1$.

\begin{conjecture}\label{conj:val21}
Let $k$, $a \in \N$ be such that $a < k$. Then $V_1 ( \{2^a,1^{2k-2a-1} \}) = B(k-1) + 1$.
\end{conjecture}

The conjecture is true when $a=0$ since $V_1( \{1^{2k-1}\}) = K(2k-1,k) = B(k-1) + 1$. Moreover, since
the position $\{2^a,1^{2k-2a-1} \}$ may arise in a Binary Knight Hunt, we have
$V_1( \{ 2^a, 1^{2k-2a-1} \} \ge B(k-1) +1$ for all~$a$. 
The conjecture has been
checked for $k \le 20$ using the program \texttt{MajorityGame.hs} 
available from the author's website. 
One motivation for the conjecture is
the following corollary which strengthens part of Theorem~\ref{thm:spiescomb}.

\begin{corollary}[Conditional on Conjecture~\ref{conj:val21}]\label{cor:cor}
Let $k$ be even and let $n = 2k-1$.
Suppose that spies are unconstrained and that a spy is known to be present.
If $k \ge 4$ then there is no questioning strategy that will both find a knight by
question $K(n,k) = n-1-B(k-1)$ 
and find a spy by question $\c{T}_S(n,k) = n-1$. 
\end{corollary}

\begin{proof}
The Spy Master should support until the Interrogator asks a question
that does not connect two singleton components. When this happens
the multiset of component sizes is $\{2^a, 1^{2k-2a-1} \}$ for some $a \in \N$.
If the question connects a component of size $2$ with a component of size $1$ then
the Spy Master accuses, and then promises the Interrogator that spies
lie in all their answers. The multiset of component weights in the resulting
majority game is $\{2^{a-1},1^{2k-2a-1} \}$.
Since $k$ is even, $B(k-2) = B(k-1) - 1$ and so, 
by the conjecture, $V_1( \{2^{a-1},1^{2k-2a-1} \} = V_1( \{2^a ,1^{2k-2a-1} \}) -1$.
The Interrogator is therefore unable to find a knight by question $K(n,k)$.

If the question connects two components of size $2$ then, as in the proof of Lemma~\ref{lemma:spiescomb74},
the Spy Master accuses if the new edge is into a source vertex, and supports if the new edge is into
a sink vertex. In the former case, the Spy Master can promise the Interrogator
that the component just created has exactly three knights and one spy, and
so corresponds to the position $\{2^{a-1},1^{n-2a} \}$ in the majority game.
The argument in the previous paragraph then applies. In the latter case, let $v$
and $v'$ be the source vertices in the new component. The Spy
Master should support on all further questions. The
Interrogator must, in some later question, ask a knight, say Person $w$, for the identity of 
either $v$ or $v'$. Suppose without loss of generality that $v$ is identified.
If $w$ and $v$ are in the same component this creates a cycle in the question graph.
Otherwise
the component $C$ of $w$ has a source vertex (other than $w$ itself, since there
are no accusations in the question graph), which the Interrogator must identify. Again
this creates a cycle. Hence the final question graph  has at least $n$ edges.
\end{proof}

\subsection{Combined games}

In the following problem we change the
victory condition in Theorem~\ref{thm:liarscomb} to combine two of
the searching games considered in this paper in a different way.
The analogous problem replacing $\c{T}_L(n,k)$ with $\n{T}_L(n,k)$ is also of interest.

\begin{problem}
Suppose that spies always lie and that a spy is known to be present. 
Consider the searching game where the Interrogator wins
if he either finds a knight by question $K(n,k) - 1$, or a spy by question $\c{T}_L(n,k)-1$.
When is this game winning for the Interrogator?
\end{problem}

Theorem~\ref{thm:spiescomb} and the conditional Corollary~\ref{cor:cor} invite
the following question.
Again the analogous problem replacing $\c{T}_S(n,k)$ with $\n{T}_S(n,k)$ is also of interest.

\begin{problem}\label{problem:spiesboth}
Suppose that spies are unconstrained and that a spy is known to be present. 
When is there a questioning strategy that
finds a knight by question $K(n,k)$ and a spy by question $\c{T}_S(n,k) = n-1$?
\end{problem}

The final paragraph of the proof of Theorem~\ref{thm:spiescomb} shows
that there is such a questioning strategy when $n = 2^e+1$ and $k = 2^{e-1}+1$, 
for any \hbox{$e \in \N$}. This reflects the ease with which the Interrogator may meet
the target $K(n,k) = n-2$ for finding a knight. We remark that Example~\ref{ex:ex}
shows one situation in which the Switching Knight Hunt is effective even when
spies are unconstrained; it might  be useful in this problem.

\def\cprime{$'$} \def\Dbar{\leavevmode\lower.6ex\hbox to 0pt{\hskip-.23ex
  \accent"16\hss}D} \def\cprime{$'$}
\providecommand{\bysame}{\leavevmode\hbox to3em{\hrulefill}\thinspace}
\providecommand{\MR}{\relax\ifhmode\unskip\space\fi MR }
\providecommand{\MRhref}[2]{%
  \href{http://www.ams.org/mathscinet-getitem?mr=#1}{#2}
}
\providecommand{\href}[2]{#2}

\end{document}